\documentclass[reqno,centertags,11pt,draft]{amsart}
\usepackage{amsmath,amsthm,amsfonts,amssymb,enumerate}
\usepackage[bookmarksopen=true,final]{hyperref}
\usepackage{mathrsfs}
\usepackage[nobysame]{amsrefs}
\usepackage{hyperref}
\usepackage[pdftex]{graphicx,color}
\usepackage{graphicx}
\usepackage{marvosym}
\usepackage{scalerel}
\usepackage{bbm}
\usepackage{bm}
\usepackage{setspace}
\usepackage{tikz}
\usepackage[foot]{amsaddr}
\usepackage[normalem]{ulem}

\usepackage{arydshln}

\usepackage{fancyhdr} 
\usepackage{ulem} 

\usepackage{xcolor}
\newtheorem{thm}{Theorem}[section]

\newtheorem{prop}[thm]{Proposition}
\newtheorem{rem}[thm]{Remark}
\newtheorem{defn}[thm]{Definition}

\newcommand{\R}{\mathbb{R}}
\newcommand{\C}{\mathbb C}

\newcommand{\Z}{\mathbb{Z}}
\newcommand{\N}{\mathbb{N}}
\newcommand{\T}{\mathbb{\partial\mathbb{D}}}
\newcommand{\D}{\mathbb{D}}

\renewcommand{\d}{\partial}
\newcommand{\supp}{\operatorname{supp}}

\DeclareMathOperator{\re}{Re}
\DeclareMathOperator{\im}{Im}

\newcommand{\abs}[1]{\left|#1\right|}
\newcommand{\set}[1]{\left\{#1\right\}}
\newcommand{\brkt}[1]{\left(#1\right)}

\begin{document}
\title{Zeros of Laurent multiple orthogonal polynomials on the unit circle %and their Hermite--Pad\'{e} problem
}
\author{Rostyslav Kozhan$^1$}
\email{kozhan@math.uu.se}
\author{Marcus Vaktnäs$^{1}$}
\email{marcus.vaktnas@math.uu.se}
\address{$^{1}$Department of Mathematics, Uppsala University, S-751 06 Uppsala, Sweden}
\date{\today}
\begin{abstract}
    We investigate two distinct formulations of Laurent multiple orthogonal polynomials on the unit circle, introduced in \cite{KVMLOPUC} and \cite{KVHP} respectively.
    %We investigate Laurent multiple orthogonal polynomials introduced in~\cite{KVMLOPUC} and in~\cite{KVHP}, for Angelesco and AT systems on the unit circle. 
    %We investigate Laurent multiple orthogonal polynomials for Angelesco and AT systems on the unit circle that were introduced in~\cite{KVMLOPUC} and in~\cite{KVHP}. 
    For the first formulation, we prove that all zeros lie strictly within the complex open unit disk for any Angelesco or AT system. For the second formulation, we establish normality
    %   The first main result proves that all zeros of polynomials from~\cite{KVMLOPUC} lie in the unit disc. The second main result establishes normality for the polynomials from ~\cite{KVHP}, 
    of all indices of the form $(\bm{n};\bm{n})$, $(\bm{n}+\bm{e}_j;\bm{n})$, and $(\bm{n};\bm{n}+\bm{e}_j)$  for any Angelesco or AT system, thereby enabling the full application of the Szeg\H{o} mapping and Geronimus relations from~\cite{KVHP} in the multiple orthogonality setting.
 
    %We investigate zeros of polynomials satisfying simultaneous orthogonality relations with respect to several measures. We prove that zeros lie in the unit disc for Angelesco and AT systems, for some polynomials whose uniqueness had been established in previous work. We also use the same zero counting approach to prove uniqueness of some closely related polynomials generated from a two-point Hermite--Pad\'e approximation. 
    %We investigate zeros of Laurent multiple orthogonal polynomials on the unit circle. Our main result is that they belong to the complex unit disc for Angelesco and AT systems.%, AT, and Nikishin systems.
\end{abstract}
\maketitle

\section{Introduction}

A fundamental question about orthogonal polynomials is the behaviour of their zeros. For measures supported on the real line (OPRL) they lie on the real line, and for measures supported on the unit circle $\d\D = \{z\in \C: |z|=1\}$ (OPUC) they lie inside the unit disc $\D = \{z\in \C: |z|<1\}$. 

The current paper investigates zeros of multiple orthogonal polynomials. They satisfy orthogonality relations with respect to a system of measures $(\mu_1,\ldots,\mu_r)$. For measures supported on the real line (MOPRL), such polynomials $P_{(n_1,\ldots,n_r)}$ have been studied extensively over the last couple of decades, initially motivated by rational approximation problems with applications in number theory. See \cite{Applications} for a short introduction with several applications, and the book of Nikishin and Sorokin \cite{bookNS} for a more thorough introduction. 

Unlike the usual orthogonal polynomials, $P_{(n_1,\dots,n_r)}$ may not have all zeros real when each measure is supported on the real line. Hence one would like to identify subclasses of systems where this complication does not enter. One example, introduced by Angelesco \cite{Angelesco}, is when the convex hulls of the supports of the measures are pairwise disjoint. One of the goals of this paper is to identify a similar result for measures supported on the unit circle (MOPUC), leading to all zeros lying inside the unit disc. 

%We also study AT systems, which Another example is given by AT systems (see e.g. \cite{bookNS}), which 

\subsection{Multiple Orthogonal Polynomials on the Unit Circle}

MOPUC was introduced by M\'{i}nguez and Van Assche in~\cite{MOPUC1}, using the following definition. 
\begin{defn}\label{def:normal A}
Let $\bm{\mu} = (\mu_1,\ldots,\mu_r)$ be a system of measures supported on $\d\D$, assumed to all have infinite support, and let $\bm{n} = (n_1,\dots,n_r)$ be a multi-index of non-negative integers. The type II multiple orthogonal polynomial $\Phi_{\bm{n}}$ is defined by the degree restriction $\deg{\Phi_{\bm{n}}} \leq \abs{\bm{n}} = n_1+\dots+n_r$ and the orthogonality relations 
\begin{equation}\label{eq:type II def A}
    \int_\T \Phi_{\bm{n}}(z)z^{-p}d\mu_j(z) = 0, \quad p = 0,\ldots,n_j-1, \quad j = 1,\dots,r.
\end{equation}
We say $\bm{n}$ is normal with respect to $\bm{\mu}$ if there exists a unique $\Phi_{\bm{n}}$ on the form
\begin{equation}\label{eq:monic polyomials def A}
    \Phi_{\bm{n}}(z) = z^{\abs{\bm{n}}} + k_{\abs{\bm{n}}-1}z^{\abs{\bm{n}}-1}+\ldots+k_{0}.
\end{equation}
\end{defn}
%In~\cite{MOPUC1,MOPUC2,KVMOPUC}, properties of such polynomials were established, including recurrence relations, compatibility conditions, a Christoffel--Darboux formula, and the associated Riemann--Hilbert and Hermite--Pad\'{e} problems. We describe the latter, as introduced in~\cite{MOPUC1}. 

Normality is necessary for almost any study of multiple orthogonality. In MOPRL, this is provided by Angelesco systems and AT systems \cite{bookNS}, for which a rather simple zero counting approach proves normality. In MOPUC however, only one example of normality was given for the polynomials $\Phi_{\bm{n}}$, in the paper \cite{MOPUC2}. Numerical simulations also show that $\Phi_{\bm{n}}$ may have zeros outside the unit disc when the measures are supported on pairwise disjoint arcs. 

\subsection{Laurent Multiple Orthogonal Polynomials on the Unit Circle}

Recently, there has been substantial development on MOPUC, through the papers \cite{KVMOPUC,HueMan,KVHP,KNik,KVMLOPUC}. In \cite{KVMLOPUC}, we identified Angelesco and AT systems on the unit circle (defined in Section \ref{normality section}), which are the analogue to the Angelesco and AT systems on the real line (see e.g. \cite{bookNS}). This was done through computations inspired by the proof of~\cite{Kui}. This approach suggested the following alternative approach to MOPUC. 
\begin{defn}\label{def:normal B}
The type II Laurent multiple orthogonal polynomial $\phi_{\bm{n}}$ is defined by $\phi_{\bm{n}}(z) \in \operatorname{span}\{z^p\}_{p = -\abs{\bm{n}}/2}^{\abs{\bm{n}}/2}$ and 
\begin{equation}\label{eq:type II def mod}
    \int \phi_{\bm{n}}(z)z^{-p}d\mu_j(z) = 0, \quad p = -n_j/2,-n_j/2+1,\ldots,n_j/2-1, \quad j = 1,\dots,r.
\end{equation}
We say that a multi-index $\bm{n}\in\N^r$ is $\phi$-normal with respect to $\bm{\mu}$ if there exists a unique $\phi_{\bm{n}}$ on the form 
\begin{equation}\label{eq:monic type II mod}
    \phi_{\bm{n}}(z) = z^{\abs{\bm{n}}/2} + \kappa_{\abs{\bm{n}}/2-1}z^{\abs{\bm{n}}/2-1} + \ldots + \kappa_{-\abs{\bm{n}}/2}z^{-\abs{\bm{n}}/2}.
\end{equation}
\end{defn}
Note that this definition requires a choice of square root function, except the case when each $n_j$ is even. %(which is perhaps the most interesting case). 
When $r = 1$, $z^{n/2}\phi_{n}(z)$ corresponds to the Szeg\H{o} polynomials $\Phi_{n}(z)$. For $r > 2$ however, normality in the sense of Definition \ref{def:normal A} is very different from normality in the sense of Definition \ref{def:normal B}. 

The terminology of $\phi$-normality is used to separate it from Definition \ref{def:normal A}, but we stress that we currently consider this to be a more useful notion of normality. In particular, the $\phi$-normality of every $\bm{n}\in\N^r$ was established in \cite{KVMLOPUC} for Angelesco and AT systems on the unit circle, while normality in the sense of Definition \ref{def:normal A} is a result we do not expect currently. The preliminaries on $\phi$-normality (including Angelesco and AT systems) is discussed further in Section \ref{normality section}. 

The main result of the paper proves that all zeros of $\phi_{\bm{n}}$ lie in $\D$ for Angelesco and AT systems. The proof goes through some functions $X_{\bm{n}}^{(\tau)}$, defined in Section \ref{multiple popuc section}, which should be viewed as the analogue to paraorthogonal polynomials \cite{JNT}. In Section \ref{zeros of multiple popuc section} we prove that all zeros of $X_{\bm{n}}^{(\tau)}$ lie on $\d\D$ using similar ideas to \cite{KVNikInt}. An application of the argument principle then proves our main result (Section \ref{zeros of mopuc section}). 

It recently came to our attention that the functions $X_{\bm{n}}^{(\tau)}$ have a close connection with the trigonometric polynomials studied in \cite{Mil,TrigMop2}. In the case of all $n_j$ even, they turn into exactly their trigonometric multiple orthogonal polynomials of semi-integer degree. In this paper they define a notion of trigonometric AT system, very closely related to our notion of AT system, and prove a normality condition. Through results of this paper, we will see that they essentially proved $\phi$-normality in AT system for the case when every $n_j$ is even, 10 years before our paper \cite{KVMLOPUC}. 

\subsection{General Two-Point Hermite--Pad\'e polynomials}

In \cite{KVHP} we took a different approach to MOPUC, through a general Hermite--Pad\'e approximation problem. The Laurent polynomials are given by the following definition.
\begin{defn}\label{def:hp polynomials}
For two multi-indices $\bm{n} \in \N^r$ and $\bm{m} \in \N^r$, we define the Laurent polynomials $\Phi_{\bm{n},\bm{m}}$ by  $\Phi_{\bm{n},\bm{m}}(z) \in \operatorname{span}\set{z^p}_{p = -\abs{\bm{m}}}^{\abs{\bm{n}}}$ and
    \begin{equation}
\label{eq:generalized hermite pade orthogonality intro}
    \int \Phi_{\bm{n},\bm{m}}(w)w^{-p}d\mu_j(w) = 0, \quad p = -m_j,\dots,n_j-1, \quad j = 1,\dots,r.
\end{equation}
The pair of multi-indices $(\bm{n},\bm{m})$ is normal if there exists a unique $\Phi_{\bm{n},\bm{m}}$ on the form
\begin{equation}\label{eq:moprlIIspan}
    \Phi_{\bm{n},\bm{m}}(z) = z^{\abs{\bm{n}}} + \ldots + \alpha_{\bm{n},\bm{m}}z^{-\abs{\bm{m}}}.
    \end{equation} 
\end{defn}
Observe that under the preceding definitions, $\Phi_{\bm{n}} = \Phi_{\bm{n},\bm{0}}$ and $\phi_{2\bm{n}} = \Phi_{\bm{n},\bm{n}}$. 
%We are mainly interested in the case $\bm{n} \approx \bm{m}$. 
Combined with the results of~\cite{KVMLOPUC}, this implies that all indices $(\bm{n},\bm{n})$ are normal for Angelesco and AT systems. A principal result in \cite{KVHP} establishes the Szeg\H{o} mapping and the Geronimus relations for multiple orthogonal polynomials, provided that indices of the form $(\bm{n},\bm{n})$, $(\bm{n}+\bm{e}_j,\bm{n})$ and $(\bm{n},\bm{n}+\bm{e}_j)$ are normal. As the second main result of the current paper, we prove that these specific indices are indeed normal for any Angelesco or AT system, thereby enabling the direct application of the Szeg\H{o} mapping and Geronimus relations derived in \cite{KVHP}.
The notion of paraorthogonal polynomials again plays an important part in the proof. 

We remark that the Laurent polynomials $\Phi_{\bm{n},\bm{n}}$, $\Phi_{\bm{n}+\bm{e}_j,\bm{n}}$ are closely related to those studied by Huertas and Ma\~nas in \cite{HueMan}, where the question of normality was left open.

\section{Normality}\label{normality section}

%Let $\bm\mu=(\mu_1,\ldots,\mu_r)$ be a system of positive measures on the unit circle each with infinite support.  

In this section we go over the basic preliminaries about normality. For the remainder of the paper we always assume the measures have infinite support. We also fix a branch of the square root function 
\begin{equation}\label{eq:sqrt}
z^{k/2}=|z|^{k/2} \exp({i k\arg_{[t_0,t_0+2\pi)}} (z)/2), \quad k \in \Z,
\end{equation}
for some chosen $t_0\in\R$. The choice will depend on the system we work with. 

\begin{rem}
    Note that $z^{k/2} = (z^{1/2})^k$, but not necessarily $z^{k/2} = (z^k)^{1/2}$. Also, $z^{-k/2} = \overline{z^{k/2}}$, but not necessarily $z^{-k/2} = \bar{z}^{k/2}$. 
\end{rem}

For $f(z) = \sum_{k = -M/2}^{N/2}c_kz^k$, where $N$ and $M$ have the same parity, we write $f^\sharp$ for the function 
\begin{equation}
    f^\sharp(z) = \overline{f(1/\bar{z})} = \sum_{k = -N/2}^{M/2}\bar{c}_{-k}z^k.
\end{equation}

\subsection{$\phi$-normality}
Note that $\phi_{\bm{n}}^\sharp(z) \in \operatorname{span}\{z^p\}_{p = -\abs{\bm{n}}/2}^{\abs{\bm{n}}/2}$ and
\begin{equation}\label{eq:type II def mod sharp}
    \int \phi_{\bm{n}}^{\sharp}(z)z^{-p}d\mu_j(z) = 0, \quad p = -n_j/2+1,\ldots,n_j/2, \quad j = 1,\dots,r.
\end{equation}
We then see that $\bm{n}$ is $\phi$-normal if and only if there is a unique $\phi_{\bm{n}}^{\sharp}$ satisfying \eqref{eq:type II def mod sharp} on the form
\begin{equation}\label{eq:monic type II mod sharp}
        \phi^\sharp_{\bm{n}}(z) = \bar\kappa_{-\abs{\bm{n}}/2}z^{\abs{\bm{n}}/2} + \ldots+ \bar\kappa_{\abs{\bm{n}}/2-1}z^{-\abs{\bm{n}}/2+1} + z^{-\abs{\bm{n}}/2}.
\end{equation} 
When $\bm{n}$ is $\phi$-normal we always assume the normalizations \eqref{eq:monic type II mod} and \eqref{eq:monic type II mod sharp}.

When solving \eqref{eq:type II def mod}-\eqref{eq:monic type II mod} or \eqref{eq:type II def mod sharp}-\eqref{eq:monic type II mod sharp} for the coefficients, we get a system of equations with matrix on the form
\begin{equation}\label{eq:moment matrix mod}
        T_{\bm{n}} =
        \left(\begin{array}{c}
            T^{(1)}_{n_1,|\bm{n}|} 
            %\\        \vdots \\
            \\[0.4em]
        \hdashline[0.5pt/2pt]
        \\[-1.2em]
         \vdots \\[0.2em]
        \hdashline[0.5pt/2pt]
        \\[-1.0em]
            T^{(r)}_{n_r,|\bm{n}|}
        \end{array}
        \right),
    \end{equation}
where $T^{(j)}_{n_j,|\bm{n}|}$ is the $n_j\times |\bm{n}|$ matrix given by
    \begin{equation}\label{eq:H1}
         T^{(j)}_{n_j,|\bm{n}|} =
         \begin{pmatrix}
            \int z^{-\abs{\bm{n}}/2}z^{n_j/2} d\mu_j(z) & \cdots & \int z^{\abs{\bm{n}}/2-1}z^{n_j/2} d\mu_j(z) \\
            \int z^{-\abs{\bm{n}}/2}z^{n_j/2-1} d\mu_j(z) & \cdots & \int z^{\abs{\bm{n}}/2-1}z^{n_j/2-1} d\mu_j(z) \\
            \vdots & \ddots & \vdots \\ 
            \int z^{-\abs{\bm{n}}/2}z^{-n_j/2+1} d\mu_j(z) & \cdots & \int z^{\abs{\bm{n}}/2-1}z^{-n_j/2+1} d\mu_j(z) \\
        \end{pmatrix}.
    \end{equation}

The next result presents some of the  conditions that are equivalent to $\phi$-normality of a multi-index $\bm{n}$.
\begin{prop}\label{prop:phi normality criteria}
    The following statements are equivalent:
    \begin{itemize}
        \item There is a unique $\phi_{\bm{n}}$ on the form \eqref{eq:monic type II mod} satisfying \eqref{eq:type II def mod}.
        \item There is no $\phi_{\bm{n}}(z) \in \operatorname{span}\{z^p\}_{p = -\abs{\bm{n}}}^{\abs{\bm{n}}-1}\setminus\{0\}$ satisfying \eqref{eq:type II def mod}.
        \item There is a unique $\phi_{\bm{n}}^\sharp$ on the form \eqref{eq:monic type II mod sharp} satisfying \eqref{eq:type II def mod sharp}.
        \item There is no $\phi_{\bm{n}}^\sharp \in \operatorname{span}\{z^p\}_{p = -\abs{\bm{n}}+1}^{\abs{\bm{n}}}\setminus\{0\}$ satisfying \eqref{eq:type II def mod sharp}.
        \item $\det(T_{\bm{n}}) \neq 0$. 
    \end{itemize}
\end{prop}

\begin{proof}
    For each of the first four statements, solving for the coefficients results in a system of equation with matrix equal to $T_{\bm{n}}$. 
\end{proof}

We remark that $\bm{0}$ is always $\phi$-normal, and the above result holds with the convention $\det(T_{\bm{0}}) = 1$.

\subsection{Angelesco systems on the unit circle}\label{ss:Angelesco}
\begin{defn}\label{def:Angelesco}
    $\bm{\mu} = (\mu_1,\dots,\mu_r)$ is an Angelesco system on $\d\D$ if there are closed arcs $\Gamma_j\subset \T$ such that $\supp\,\mu_j\subseteq \Gamma_j$, $j = 1,\dots,r$, and $\Gamma_j$ intersects $\Gamma_k$ at no more than the endpoints of the arcs, whenever $j \neq k$.
\end{defn}
%\begin{rem}
%    In the above definition, it is easy to see that $\Gamma_j\cap \Gamma_k$ can consist of two points only if $r=2$ and the union of $\Gamma_1$ and $\Gamma_2$ is equal to $\T$.
%\end{rem}
We assume that $t_0$ and the order of $\Gamma_j$'s is chosen so that $t_0 \le  \theta_1 \le \dots \le \theta_r \le t_0 + 2\pi$ holds whenever $e^{i\theta_1} \in \Gamma_1,\dots,e^{i\theta_r} \in \Gamma_r$. We choose the branch of the square root function to be~\eqref{eq:sqrt} with this  $t_0$. 
Let us additionally require that $\mu_r$ does not have a point mass at $e^{it_0}$.

By direct computation of the determinant of the moment matrix $T_{\bm{n}}$ we proved the following result. 

\begin{thm}[\cite{KVMLOPUC}]\label{thm:Angelesco}
    If $\bm{\mu}$ is an Angelesco system on the unit circle then every $\bm{n} \in \N^r$ is $\phi$-normal.
    %Angelesco systems on the unit circle are perfect.
\end{thm}

\subsection{AT systems on the Unit Circle}\label{ss:AT}
\hfill\\

    A collection of continuous real-valued functions $\set{u_j(t)}_{j=1}^n$ on an interval $[a,b]$ is a Chebyshev system on $[a,b]$ if the determinants
    \begin{equation}\label{eq:Chebyshev}
        W_n(\bm{x}):=\det
        \begin{pmatrix}
            u_1(x_1) & u_1(x_2) & \cdots & u_1(x_n) \\
            u_2(x_1) & u_2(x_2) & \cdots & u_2(x_n) \\
            \vdots & \vdots & \ddots & \vdots \\
            u_n(x_1) & u_n(x_2) & \cdots & u_n(x_n) \\
        \end{pmatrix} \ne 0
    \end{equation}
    for every choice of different points $x_1,\ldots,x_n\in[a,b]$. By continuity, the sign of the determinant in~\eqref{eq:Chebyshev} is always positive or always negative, if we additionally introduce the ordering $x_1<x_2<\ldots<x_n$. %As long as $W_n(\bm{x})$ does not change sign, we may extend the definition of Chebyshev system to allow $W_n(\bm{x}) = 0$ on a measure zero set, without breaking the results below. 

    For a function $f(\theta)$ we use the notation
    \begin{equation}
        \operatorname{Trig}_m(f) = \begin{cases}
            \set{f(\theta)\cos{\frac{2k-1}{2}\theta},f(\theta)\sin{\frac{2k-1}{2}\theta}}_{k = 1}^{m/2}, & \mbox{ if } m \mbox{ is even},
            \\
            \set{1}\cup\set{f(\theta)\cos{k\theta},f(\theta)\sin{k\theta}}_{k = 1}^{(m-1)/2}, & \mbox{ if } m \mbox{ is odd}.
        \end{cases}
    \end{equation}

    \begin{defn}\label{def:AT}
        Assume $d\mu_j(e^{i\theta})=w_j(\theta)d\mu(e^{i\theta})$, $j = 1,\dots,r$, for a measure $\mu$ supported on the arc $\Gamma = \{e^{i\theta}:\alpha\le \theta\le \beta\}$ with $0<\beta-\alpha\le 2\pi$. Then $\bm{\mu}$ is an AT system on $\Gamma$ for the index $\bm{n} = (n_1,\ldots,n_r)$ if 
        \begin{equation}\label{eq:AT}
            \operatorname{Trig}_{\bm{n}}(\bm{\mu}) = \bigcup_{j=1}^r \operatorname{Trig}_{n_j}(w_j)
        \end{equation}
        is a Chebyshev system on $[\alpha,\beta]$.
    \end{defn}
    %In particular, AT systems on the the whole $\T$ corresponds to $\alpha=0$, $\tau=2\pi$. Note that even though we are on the circle, we do not require the {\it periodic} Chebyshev property (which is only defined for systems with an odd number of elements, see~\cite{KarStu,KreNud}), since, in particular, cosine/sine of the half-angle are not $2\pi$-periodic, and $w_j$ can have a discontinuity at a point.

    For AT system on the unit circle, the choice $t_0 = \alpha$ is natural. Through direct computation of the $\det{T_{\bm{n}}}$ we proved the following result. 
    
    \begin{thm}[\cite{KVMLOPUC}]\label{thm:AT}
        If $\bm{\mu}$ is an AT system on the unit circle then every $\bm{n} \in \N^r$ is $\phi$-normal.
    \end{thm}

    In \cite{KNik}, the AT property was verified for Nikishin systems on the unit circle, hence providing a large class of systems for which $\phi$-normality holds for a large set of indices.

\subsection{$\Phi$-normality}

We now discuss the polynomials from Definition \ref{def:hp polynomials}. We have the following related definition. 

\begin{defn}
    The Laurent polynomials $\Phi_{\bm{n},\bm{m}}^*$ are defined by the degree restriction $\Phi_{\bm{n},\bm{m}}^*(z) \in \operatorname{span}\set{z^p}_{p = -\abs{\bm{m}}}^{\abs{\bm{n}}}$ and the orthogonality relations
    \begin{equation}
\label{eq:generalized hermite pade orthogonality star}
    \int \Phi_{\bm{n},\bm{m}}^*(w)w^{-p}d\mu_j(w) = 0, \quad p = -m_j+1,\dots,n_j, \quad j = 1,\dots,r.
\end{equation}
\end{defn}
For the polynomials $\Phi_{\bm{n},\bm{m}}^*$ we want the normalization
\begin{equation}\label{eq:moprlIIspanStar}
    \Phi_{\bm{n},\bm{m}}^*(z) = \beta_{\bm{n},\bm{m}}z^{\abs{\bm{n}}} + \ldots + z^{-\abs{\bm{m}}}.
\end{equation}
The analogue of Proposition \ref{prop:phi normality criteria} for $\Phi$-normality is the following result. 
\begin{prop}\label{prop:normality criteria hp}
    Normality of $(\bm{n},\bm{m})$ is equivalent to the following statements.
    \begin{itemize}
        \item There is a unique $\Phi_{\bm{n},\bm{m}}$ on the form \eqref{eq:moprlIIspan} satisfying \eqref{eq:generalized hermite pade orthogonality intro}.
        \item There is no $\Phi_{\bm{n},\bm{m}}(z) \in \operatorname{span}\{z^p\}_{p = -\abs{\bm{m}}}^{\abs{\bm{n}}-1}\setminus\{0\}$ satisfying \eqref{eq:generalized hermite pade orthogonality intro}.
        \item There is a unique $\Phi_{\bm{n},\bm{m}}^*$ on the form \eqref{eq:moprlIIspanStar} satisfying \eqref{eq:generalized hermite pade orthogonality star}.
        \item There is no $\Phi_{\bm{n},\bm{m}}^* \in \operatorname{span}\{z^p\}_{p = -\abs{\bm{m}}+1}^{\abs{\bm{n}}}\setminus\{0\}$ satisfying \eqref{eq:generalized hermite pade orthogonality star}.
    \end{itemize}
\end{prop}

\begin{proof}
    Similarly to Proposition \ref{prop:phi normality criteria}, these items are associated with the same matrix. 
\end{proof}

We also have the following proposition, relating $\Phi_{\bm{n},\bm{m}}^*$ and $\Phi_{\bm{n},\bm{m}}^\sharp$. 

\begin{prop}\label{prop:n,m normality vs m,n normality}
    $(\bm{n},\bm{m})$ is normal if and only if $(\bm{m},\bm{n})$ is normal, and 
    \begin{equation}
        \Phi_{\bm{n},\bm{m}}^\sharp = \Phi_{\bm{m},\bm{n}}^*.
    \end{equation}
\end{prop}

\begin{proof}
    Simply check the orthogonality relations of $\Phi_{\bm{n},\bm{m}}^\sharp$ and $\Phi_{\bm{m},\bm{n}}^*$ and apply Proposition \ref{prop:normality criteria hp}.
\end{proof}

\section{Multiple Paraorthogonal Polynomials}\label{multiple popuc section}

The results of the paper will be centered around the following definition. 

\begin{defn}\label{def:para}
For any $\phi$-normal index $\bm{n}\in\N^r$ and any $\tau\in\T$, the type II paraorthogonal polynomial $X_{\bm{n}}^{(\tau)}$ is defined by
\begin{equation}\label{eq:paraorthogonal polynomials defn}
    % X_{\bm{n}}^{(\tau)}(z) = \tau^{1/2} z^{1/2}\phi_{\bm{n}}(z) + \tau^{-1/2} z^{-1/2}\phi_{\bm{n}}^\sharp(z).
    X_{\bm{n}}^{(\tau)}(z) = z^{1/2}\phi_{\bm{n}}(z) + \tau z^{-1/2}\phi_{\bm{n}}^\sharp(z).
\end{equation}
\end{defn}
\begin{rem}
    Note that if $r=1$ then 
\begin{equation}
    z^{(n+1)/2} X_{n}^{(\tau)}(z) =   z\Phi_{n}(z) + \tau \Phi_{n}^*(z),
\end{equation} 
which is the usual paraorthogonal polynomials $\Phi_{n+1}^{(\tau)}$ from the classical theory.
\end{rem}

%Note that if $n=1$ then this becomes $X_{\bm{n}}^{(\tau)}(z) =  \tau^{1/2} z^{-n/2+1/2}\Phi_{\bm{n}}(z) + \tau^{-1/2} z^{-n/2-1/2}\Phi_{\bm{n}}^*(z)$, i.e., $ z^{(n+1)/2} X_{\bm{n}}^{(\tau)}(z) =  \tau^{1/2} z\Phi_{\bm{n}}(z) + \tau^{-1/2} \Phi_{\bm{n}}^*(z) = \tau^{1/2} (z\Phi_n(z) + \bar{\tau} \Phi^*_{n}(z))$, which is effectively the usual definition of $\Phi_{n+1}^{(\tau)}$.

%Note that if $n=1$ then this becomes $X_{\bm{n}}^{(\tau)}(z) =  z^{-n/2+1/2}\Phi_{\bm{n}}(z) + \tau z^{-n/2-1/2}\Phi_{\bm{n}}^*(z)$, i.e., $ z^{(n+1)/2} X_{\bm{n}}^{(\tau)}(z) =   z\Phi_{\bm{n}}(z) + \tau \Phi_{\bm{n}}^*(z)$, which is effectively the usual definition of $\Phi_{n+1}^{(\tau)}$.

By \eqref{eq:paraorthogonal polynomials defn}, we have $X_{\bm{n}}^{(\tau)}(z) \in \operatorname{span}\set{z^p}_{p = -(\abs{\bm{n}}+1)/2}^{(\abs{\bm{n}}+1)/2}$ and the symmetric orthogonality relations
\begin{equation}\label{eq:orthoPOPUC}
    \int X_{\bm{n}}^{(\tau)}(z)z^{-p}d\mu_j(z) = 0, \quad p = -(n_j-1)/2,\dots,(n_j-1)/2.
\end{equation}
We also have
\begin{equation}\label{eq:beta invariance}
    X_{\bm{n}}^{(\tau)}(z) = \tau\overline{X_{\bm{n}}^{(\tau)}(1/\bar{z})},
\end{equation}
which we will refer to as the $\tau$-invariance of $X_{\bm{n}}^{(\tau)}$. %\red{Note that \eqref{eq:orthoPOPUC}--\eqref{eq:beta invariance} is always satisfied by
%\begin{equation}
%    az^{1/2}\phi_{\bm{n}}(z) + \bar{a}\tau z^{-1/2}\phi_{\bm{n}}^\sharp(z), \quad a \in \C.
%\end{equation}
%}

\begin{prop}\label{prop:popuc normality}
    $\bm{n}$ is $\phi$-normal if and only if there is no non-trivial $\tau$-invariant $X(z) \in \operatorname{span}\set{z^p}_{p = -(\abs{\bm{n}}-1)/2}^{(\abs{\bm{n}}-1)/2}$ satisfying \eqref{eq:orthoPOPUC}. 
\end{prop}

\begin{proof}
    The matrix of \eqref{eq:orthoPOPUC} for $X(z) \in \operatorname{span}\set{z^p}_{p = -(\abs{\bm{n}}-1)/2}^{(\abs{\bm{n}}-1)/2}$ is exactly equal to $T_{\bm{n}}$. This is immediate from the orthogonality relations and degree restriction of $\phi = z^{-1/2}X(z)$. We can now build the $\tau$-invariant $X + \tau X^\sharp$. If $X + \tau X^\sharp \not\equiv 0$ we are done. Otherwise, we have $X = -\tau X^\sharp$, in which case $iX$ is $\tau$-invariant. 
\end{proof}

%Note that $X_{\bm{n}}^{(\tau)}(z) = \tau\overline{X_{\bm{n}}^{(\tau)}(1/\bar{z})}$. From Proposition \ref{prop:popuc normality} it follows that $\bm{n}$ is $\phi$-normal if and only if for every $\tau \in \d\D$ there is a unique polynomial $X_{\bm{n}}^{(\tau)}$ up to multiplication by a real constant, with $X_{\bm{n}}^{(\tau)}(z) = \tau\overline{X_{\bm{n}}^{(\tau)}(1/\bar{z})}$ and $\deg{z^{(\abs{n}+1)/2}X_{\bm{n}}^{(\tau)}(z)} = \abs{\bm{n}}+1$ such that \eqref{eq:orthoPOPUC} holds (in which case $X_{\bm{n}}^{(\tau)}$ is given by real multiples of \eqref{eq:paraorthogonal polynomials defn}).

As an alternative to Definition \ref{def:para}, we can work with the function
\begin{equation}
    T_{\bm{n}}^{(\tau)}(\theta)
    =\tfrac12 \tau^{-1/2}X_{\bm{n}}^{(\tau)}(e^{i\theta})
    =
    \tfrac12\tau^{-1/2} z^{1/2}\phi_{\bm{n}}(z) + \tfrac12\tau^{1/2} z^{-1/2}\phi_{\bm{n}}^\sharp(z)
    .
\end{equation}
Note that $\tau^{-1/2}X_{\bm{n}}^{(\tau)}(z)$ is self-reciprocal since $\tau\in\T$, so $T_{\bm{n}}^{(\tau)}(\theta)$ is a real trigonometric polynomial
\begin{equation}\label{eq:trigPOPUC}
    T_{\bm{n}}^{(\tau)}(\theta)
    =
    \begin{cases}
         %\sum_{j=0}^{|\bm{n}|/2}  \left( a_j \cos\left(j+\tfrac12\right) \theta +b_j\sin\left(j+\tfrac12\right)\theta \right),   & \mbox{if } |\bm{n}| \mbox{ is even},
         \sum_{j=1/2}^{(|\bm{n}|+1)/2}  \left( a_j \cos j\theta +b_j\sin j\theta \right)
         & \mbox{if } |\bm{n}| \mbox{ is even},
         \\
          a_0+ \sum_{j=1}^{(|\bm{n}|+1)/2} \left( a_j \cos j\theta +b_j\sin j\theta \right),  & \mbox{if } |\bm{n}| \mbox{ is odd},
    \end{cases}
\end{equation}
with
\begin{equation}\label{eq:trigPOPUC2}
    a_{(|\bm{n}|+1)/2} = \re(\tau^{1/2})
    ,
    \quad
    b_{(|\bm{n}|+1)/2} =\im(\tau^{1/2}).
\end{equation}
%where $\theta\in [t_0,t_0+2\pi)$.

%Note that if $\tau=1$ then, up to a multiplication by $-1$, we have $a_{(|\bm{n}|+1)/2} = 1$ and $b_{(|\bm{n}|+1)/2} = 0$, so that $T_{\bm{n}}^{(\tau)}(\theta)$ has only cosine as its highest term, and if $\tau=-1$ then, up to a multiplication by $-1$,  $a_{(|\bm{n}|+1)/2} = 0$, $b_{(|\bm{n}|+1)/2} = 1$, so that $T_{\bm{n}}^{(\tau)}(\theta)$ has only sine as its highest term.

%If we denote the right-hand side of~\eqref{eq:trigPOPUC} by $T_{\bm{n}}^{(\tau)}(\theta)$, then
The orthogonality conditions~\eqref{eq:orthoPOPUC} can be rewritten in the trigonometric form
\begin{align}\label{eq:orthoTrigPOPUC}
    \int_{t_0}^{t_0+2\pi} T_{\bm{n}}^{(\tau)}(\theta) \cos p\theta \, d\mu_j(e^{i\theta}) = 0, \quad 
    p = -(n_j-1)/2,\dots,(n_j-1)/2,\\
    \label{eq:orthoTrigPOPUC2}
    \int_{t_0}^{t_0+2\pi} T_{\bm{n}}^{(\tau)}(\theta) \sin p\theta \, d\mu_j(e^{i\theta}) = 0, \quad p = -(n_j-1)/2,\dots,(n_j-1)/2,
\end{align}
Such multiple orthogonality for trigonometric polynomials has been studied in~\cite{Mil}, in the case of all even $n_j$.

%\section{Paraorthogonal polynomials, type I}

%\section{Location of zeros, type II}

\section{Zeros of Multiple Paraorthogonal Polynomials}\label{zeros of multiple popuc section}

%I checked this section, it seems right. 

%But I think phrasing everything in terms of $T_{\bm{n}}$, rather than $X_{\bm{n}}$ is cleaner. You can just go through our ~\cite{KVNikInt} paper and get straightforward analogues. No need to worry about square roots (apart from the previous section). For Angelesco we can use our approach \cite[Prop. 4.6--Cor.4.7]{KVNikInt} which I think is nicer (and of independent interest, unless already known), than counting zeros. The remaining proofs (type I and interlacing) will probably also be nicer via trigonometric approach. Then in the end we state corollaries in terms of zeros of $\phi$'s.

In this section we discuss zeros of the polynomials $z^{(\abs{\bm{n}}+1)/2}X_{\bm{n}}^{(\tau)}(z)$. With a slight abuse of terminology we will refer to them as the zeros of $X_{\bm{n}}^{(\tau)}$.
Here we write $X_{\bm{n}}^{(\tau)}(z)$ for any $\tau$-invariant function in $\operatorname{span}\set{z^p}_{p = -(\abs{\bm{n}}+1)/2}^{(\abs{\bm{n}}+1)/2}$ satisfying \eqref{eq:orthoPOPUC}. Theorem \ref{thm:zeros outside circle} and Theorem \ref{thm:simple zeros on circle} are analogue to results on the real line from \cite{KVNikInt,KV1}. The first result gives a criterion for no zeros to lie outside the unit circle.

\begin{thm}\label{thm:zeros outside circle}
    Assume $\bm{n}$ is $\phi$-normal for $\bm{\mu}$. Then $X_{\bm{n}}^{(\tau)}$ has all $|\bm{n}|+1$ zeros on the unit circle if $\bm{n}$ is $\phi$-normal with respect to the systems %If $\abs{z_0} \neq 1$, then $z_0 = r_0e^{i\theta_0}$ is a zero of $X_{\bm{n}}^{(\tau)}$ if and only if $\bm{n}$ is not normal with respect to 
    \begin{equation}
        d\widehat{\bm{\mu}}(z) = \abs{z-z_0}^2d\bm{\mu}(z), \quad 0 \neq \abs{z_0} \neq 1.%= (1 + r_0^2 - 2r_0\cos(\theta - \theta_0))d\bm{\mu}(e^{i\theta}), \quad z_0 = r_0e^{i\theta_0}, \quad 0 \neq r_0 \neq 1.
    \end{equation}
\end{thm}

\begin{proof}
    First, $z^{(\abs{\bm{n}}+1)/2}X_{\bm{n}}^{(\tau)}(z)$ cannot vanish at $0$, since we would then have $X_{\bm{n}}^{(\tau)} \in \operatorname{span}\set{z^p}_{p = -(\abs{\bm{n}}-1)/2}^{(\abs{\bm{n}}-1)/2}$ by $\tau$-invariance. This contradicts $\phi$-normality of $\bm{n}$ for $\bm{\mu}$, by Proposition \ref{prop:popuc normality}.
    
    Now suppose $X_{\bm{n}}^{(\tau)}(z_0) = 0$ and $0 \neq \abs{z_0} \neq 1$. Then also $X_{\bm{n}}^{(\tau)}(1/\bar{z}_0) = 0$. Hence we can pull out the factor $z^{-1}(z-z_0)(z-1/\bar{z}_0)$ from $X_{\bm{n}}^{(\tau)}$. Note that 
\begin{align}\label{eq:factorization 1}
    z^{-1}(z-z_0)(z-1/\bar{z}_0) = -\bar{z}_0^{-1}\abs{z-z_{0}}^2, \quad z \in \d\D.
\end{align}
We write $X_{\bm{n}}^{(\tau)}(z) = \abs{z-z_{0}}^2Y(z)$ for $z \in \d\D$, which defines a Laurent polynomial $Y(z) \in \operatorname{span}\set{z^p}_{p = -(\abs{\bm{n}}-1)/2}^{(\abs{\bm{n}}-1)/2}$. Then $z^{-1/2}Y(z)$ satisfies all the orthogonality relations with respect to $\widetilde{\bm{\mu}}$, but the degree restriction contradicts $\phi$-normality since $Y(z) \not\equiv 0$.

%Conversely, suppose $\bm{n}$ is not $\phi$-normal with respect to $\bm{\widetilde{\mu}} = \abs{z-z_{0}}^2\bm{\mu}$ and $0 \neq \abs{z_0} \neq 1$. Then there is some $\widetilde{X}_{\bm{n}}^{(\tau)}(z) \in \operatorname{span}\set{z^p}_{p = -(\abs{\bm{n}}-1)/2}^{(\abs{\bm{n}}-1)/2}$ with respect to $\bm{\widetilde{\mu}}$, by Proposition \ref{prop:popuc normality}. Now put%multiple orthogonal polynomial $\widetilde{\phi}_{\bm{n}}(z) \in \operatorname{span}\set{z^p}_{p = -\abs{\bm{n}}/2}^{\abs{\bm{n}}/2-1}$ with respect to $\bm{\widetilde{\mu}}$. Now put
%\begin{equation}
%    X(z) = e^{-i\theta_0}z^{-1}(z-z_0)(z-1/\bar{z}_0)(z^{1/2}\widetilde{\phi}_{\bm{n}}(z) + \tau z^{-1/2}\overline{\widetilde{\phi}_{\bm{n}}(1/\bar{z})}). 
%\end{equation}
%\begin{equation}
%    X(z) = e^{-i\theta_0}z^{-1}(z-z_0)(z-1/\bar{z}_0)\widetilde{X}_{\bm{n}}^{(\tau)}(z). 
%\end{equation}
%Note that $X(z) = \tau\overline{X(1/\bar{z})}$. From the orthogonality relations of $\widetilde{X}_{\bm{n}}^{(\tau)}$, it then follows that $X_{\bm{n}}^{(\tau)}$ is a non-zero real multiple of $X$. Hence $X_{\bm{n}}^{(\tau)}(z_0) = 0$. 
\end{proof}

The next two results concern the behaviour of zeros that lie on the circle. 

\begin{thm}\label{thm:simple zeros on circle}
    Assume $\bm{n}$ is $\phi$-normal for $\bm{\mu}$. Then  any zero of $X_{\bm{n}}^{(\tau)}$ on $\d\D$ is simple if $\bm{n}$ is $\phi$-normal with respect to the systems 
    \begin{equation}
        d\widehat{\bm{\mu}}(z) = -e^{i\varphi}z^{-1}(z-e^{i\varphi})^2d\bm{\mu}(z) = 4\sin^2\frac{\theta-\varphi}{2}d\bm{\mu}(e^{i\theta}), \quad \varphi \in [0,2\pi).
    \end{equation}
\end{thm}

\begin{proof}
    First note that   \begin{equation}\label{eq:factorization 2}
    -ie^{-i\varphi/2}z^{-1/2}(z-e^{i\varphi}) = 2\sin\frac{\theta-\varphi}{2} \in \R, \quad z \in \d\D.
    \end{equation}
    Now,  $X_{\bm{n}}^{(\tau)}$ has a zero $e^{i\varphi}$ of higher multiplicity if and only if 
    \begin{equation}
        X_{\bm{n}}^{(\tau)}(z) = e^{i\varphi}z^{-1}(z-e^{i\varphi})^2Y(z), \quad Y(z) = \tau\overline{Y(1/\bar{z})} \in \operatorname{span}\set{z^p}_{p = -(\abs{\bm{n}}-1)/2}^{(\abs{\bm{n}}-1)/2}.
    \end{equation}
    The result then follows by Proposition \ref{prop:popuc normality} applied to $\widehat{\bm{\mu}}$.
\end{proof}

\begin{thm}\label{thm:two zeros the circle}
    Assume $\bm{n}$ is $\phi$-normal for $\bm{\mu}$. If $X_{\bm{n}}^{(\tau)}$ has two zeros $e^{i\varphi_1}$ and $e^{i\varphi_2} \neq e^{i\varphi_1}$ then $\bm{n}$ is not $\phi$-normal with respect to 
    \begin{multline}\label{eq:double christoffel transform on circle}
         d\widehat{\bm{\mu}}(z) = e^{-\varphi_1/2}z^{-1/2}(z-e^{i\varphi_1})e^{-\varphi_2/2}z^{-1/2}(z-e^{i\varphi_2})d\bm{\mu}(z) \\ = 4\sin\frac{\theta-\varphi_1}{2}\sin\frac{\theta-\varphi_2}{2}d\bm{\mu}(e^{i\theta}).
    \end{multline}
\end{thm}

\begin{proof}
    Same arguments as in the proof of Theorem \ref{thm:simple zeros on circle}.
\end{proof}

%\begin{cor}
%    All zeros of $X_{\bm{n}}^{(\tau)}$ are simple and lie on $\d\D$ and all zeros of $T_{\bm{n}}^{(\tau)}$ are real and simple in Angelesco and AT systems.
%\end{cor}

For Angelesco and AT systems, Theorem \ref{thm:zeros outside circle}-\ref{thm:two zeros the circle} gather into the following results. 

\begin{thm}\label{thm:at zeros}
    If $\bm{\mu}$ is an AT system, then all $|\bm{n}|+1$ zeros of $X_{\bm{n}}^{(\tau)}$ are simple and lie on the unit circle. At most one zero lies outside of $\operatorname{Int}(\Gamma)$.
\end{thm}
\begin{proof}
    The AT property for $\abs{z-z_0}^2d\bm{\mu}$ and $\sin^2\frac{\theta-\varphi}{2}d\bm{\mu}(e^{i\theta})$ clearly holds, so all zeros are simple and unimodular. If $e^{i\varphi_1} \neq e^{i\varphi_2}$ are zeros outside $\operatorname{Int}(\Gamma)$, then $\sin\frac{\theta-\varphi_1}{2}\sin\frac{\theta-\varphi_2}{2}$ keeps the same sign on $\operatorname{Int}(\Gamma)$. Hence Theorem \ref{thm:two zeros the circle} applies since $\abs{\sin\frac{\theta-\varphi_1}{2}\sin\frac{\theta-\varphi_2}{2}}d\bm{\mu}(e^{i\theta})$ is AT with the same multiple orthogonal polynomials as \eqref{eq:double christoffel transform on circle}. 
\end{proof}

\begin{thm}\label{thm:angelesco zeros}
    If $\bm{\mu}$ is an Angelesco system, then all $|\bm{n}|+1$ zeros of $X_{\bm{n}}^{(\tau)}$ are simple and lie on the unit circle. At most one zero lies  outside $\bigcup_{j = 1}^r\operatorname{Int}(\Gamma_j)$.
\end{thm}

\begin{proof}
    Same as the proof of Theorem \ref{thm:at zeros}.
\end{proof}

\begin{rem}
    In this paper, we currently focus strictly on the zero location for type II Laurent multiple orthogonal polynomials. The analogous analysis for type I polynomials based on~\cite{KVNikInt} produces similar results. This is the subject of ongoing research by the authors and will be addressed in subsequent work.
    %In this paper, we focus strictly on the zero location and normality on the unit circle for type II Laurent multiple orthogonal polynomials. The analogous analysis for type I polynomials based on~\cite{KVNikInt} is the subject of ongoing research by the authors and will be addressed in subsequent work. 
\end{rem}

Theorem \ref{thm:angelesco zeros} can be improved, by the following result, based on the standard zero counting proof for multiple orthogonal polynomials on the real line (see e.g. \cite{Ismail}). 

\begin{thm}\label{thm:zero counting}
    If $\bm{n}$ is normal, then $X_{\bm{n}}^{(\tau)}$ has at least $n_j$ zeros of odd multiplicity in $\operatorname{Int}(\Gamma_j)$. 
\end{thm}
\begin{proof}
By \eqref{eq:factorization 1} and \eqref{eq:factorization 2} we have, up to multiplication by a constant, the factorization
\begin{equation}\label{eq:factorization structure}
    X_{\bm{n}}^{(\tau)}(z) = \prod_{j = 1}^k\abs{z-z_j}^2 \prod_{j = 1}^{\abs{\bm{n}}-2k+1} ie^{-i\varphi_j/2}z^{-1/2}(z-e^{i\varphi_j}), \quad z \in \d\D.
\end{equation}

Let $e^{i\theta_1},\dots,e^{i\theta_\ell}$ be the zeros of $X_{\bm{n}}^{(\tau)}$ in $\operatorname{Int}(\Gamma_j)$ of odd multiplicity. Then we can write $X_{\bm{n}}^{(\tau)}(z) = Y(z)L(z)$, where $Y(z)$ is real valued and does not change sign on $\Gamma_j$, and $L(z) = \prod_{j = 1}^{\ell}ie^{i\varphi_j/2}z^{-1/2}(z-e^{i\varphi_j})$. 

Assume first that $\ell$ and $n_j$ have different parity, so that if $\ell < n_j$ then $L(z) \in \operatorname{span}\set{z^p}_{p = -(n_j-1)/2}^{(n_j-1)/2}$. Since $X_{\bm{n}}^{(\tau)}(z)L(z)$ does not change sign on $\Gamma_k$ we have
\begin{equation}
    \int X_{\bm{n}}^{(\tau)}(z)L(z)d\mu_j(z) \neq 0.
\end{equation}
By the orthogonality relations \eqref{eq:orthoPOPUC}, we must have $\ell \geq n_j$. 

Now assume that $\ell$ and $n_j$ have the same parity. Then if $\ell < n_j$ we have $z^{1/2}L(z) \in \operatorname{span}\set{z^p}_{p = -(n_j-1)/2}^{(n_j-1)/2}$. Note that the imaginary part of $e^{-it_0/2}X_{\bm{n}}^{(\tau)}(z)z^{1/2}L(z)$ does not change sign on $\Gamma_j$. Hence we have
\begin{equation}
    \int X_{\bm{n}}^{(\tau)}(z)z^{1/2}L(z)d\mu_j(z) \neq 0.
\end{equation}
We conclude that $\ell \geq n_j$.
\end{proof}

The factorization \eqref{eq:factorization structure} on trigonometric form is given by
\begin{equation}
    T_{\bm{n}}^{(\tau)}(\theta) = \prod_{j=1}^k(1+r_j^2-2r_j\cos(\theta-\theta_j))\prod_{j=1}^{\abs{\bm{n}}-2k+1}\sin{\frac{\theta-\varphi_j}{2}}. 
\end{equation}
The above proof can be translated to the trigonometric point of view without substantial changes. A zero counting proof for AT systems in the case when every $n_j$ is even can be found in \cite{Mil}. We also remark that the above zero counting proof is a variation of a result from \cite{quasi popuc}, which locates zero of quasi-paraorthogonal polynomials. 

\begin{rem}
    Theorem \ref{thm:zero counting} verifies the criterion of Proposition \ref{prop:popuc normality} for Angelesco systems. Hence we have found a second proof of $\phi$-normality. A similar zero counting approach for AT systems, extending the proof of \cite{Mil} to any $\bm{n}$, gives a second proof of $\phi$-normality of AT systems as well. 
\end{rem}

\section{Zeros of Multiple Orthogonal Polynomials}\label{zeros of mopuc section}

Through the results of the previous section we are now ready to prove the main results of the paper. 

\begin{thm}\label{thm:zeros in unit disc}
    For Angelesco and AT systems, the polynomial $z^{\abs{\bm{n}}/2}\phi_{\bm{n}}(z)$ has $\abs{\bm{n}}$ zeros in $\D$, with each zero counted as many times as its multiplicity. 
\end{thm}

%\begin{thm}
%    All $\abs{\bm{n}}$ zeros of $\phi_{\bm{n}}$ lie in $\D$ for any Angelesco or AT system. 
%\end{thm}

\begin{proof}
    In this proof we work with $X_{\bm{n}}^{(\tau)}$ on the form \eqref{eq:paraorthogonal polynomials defn}. We first show that $P(z) = z^{\abs{\bm{n}}/2}\phi_{\bm{n}}(z)$ has no zeros on $\d\D$. If we have $\phi_{\bm{n}}(e^{i\varphi}) = 0$ for some angle $\varphi$ then we would have $X_{\bm{n}}^{(\tau)}(e^{i\varphi}) = 0$ for every $\tau \in \d\D$. Note that $e^{i\varphi}$ cannot be a double zero of $\phi_{\bm{n}}$, since then it is also a double zero of $X_{\bm{n}}^{(\tau)}$, which contradicts Theorem \ref{thm:at zeros}-\ref{thm:angelesco zeros}. For the polynomial $P^{(\tau)}(z) = z^{(\abs{\bm{n}}+1)/2}X_{\bm{n}}^{(\tau)}(z)$, we have
    \begin{equation}
        \frac{d}{d\theta}P^{(\tau)}(e^{i\theta})\Bigg|_{\theta = \varphi} = e^{i\varphi}\frac{d}{d\theta}P(e^{i\theta})\Bigg|_{\theta = \varphi} + \tau\frac{d}{d\theta}P^*(e^{i\theta})\Bigg|_{\theta = \varphi},
    \end{equation}
    where $P^*(z) = z^{\abs{\bm{n}}}\phi_{\bm{n}}^\sharp(z) = z^{\abs{\bm{n}}}\overline{P(1/\bar{z})}$. Hence we get a double zero of $P^{(\tau)}$ with the choice $\tau = -e^{i(1-\abs{\bm{n}})\varphi}\frac{d}{d\theta}P(e^{i\theta})/\frac{d}{d\theta}\overline{P(e^{i\theta})} \in \d\D$, which again contradicts Theorem \ref{thm:at zeros}-\ref{thm:angelesco zeros}. Hence all zeros of $P$ lie in $\C\setminus\d\D$.

    Now consider the Blaschke product
    \begin{equation}
        B(z) = \frac{zP(z)}{P^*(z)} = z\prod_{j = 1}^{\abs{\bm{n}}}\frac{z-z_j}{1-\bar{z}_j z}.
    \end{equation}
    On the unit circle we have
    \begin{equation}
        B(e^{i\theta}) = e^{i\theta}\prod_{j = 1}^{\abs{\bm{n}}}\frac{e^{-i\theta}(e^{i\theta}-z_j)^2}{\abs{e^{i\theta}-z_j}^2} = e^{i\Psi(\theta)}, \quad \theta \in [-\pi,\pi),
    \end{equation}
    where $\Psi$ is the continuous phase function
    \begin{equation}
        \Psi(\theta) = \theta + 2\sum_{j = 1}^{\abs{\bm{n}}}\brkt{-\frac{\theta}{2} + \operatorname{Arg}(e^{i\theta}-z_j)}.
    \end{equation}
    By the argument principle we have 
    \begin{equation}
        \Psi(\pi-) - \Psi(-\pi) = 2\pi(n_++1-n_-), 
    \end{equation}
    where $n_+$ the number of zeros of $P$ in $\D$ and $n_-$ is the number of zeros of $P$ outside $\D$. 

    Note that $B(e^{i\theta}) = -\tau$ when $P^{(\tau)}(e^{i\theta}) = 0$, in which case $\Psi(\theta) = \operatorname{Arg}(-\tau) + 2\pi k$ for some integer $k$. We claim that $\Psi$ is strictly monotone, so that we would get $\Psi(\pi-) - \Psi(-\pi) = \pm2\pi (\abs{\bm{n}}+1)$, as $P^{(\tau)}$ has $\abs{\bm{n}}+1$ simple zeros on $\d\D$. Since $P$ has degree $\abs{\bm{n}}$, this would imply $n_+ = \abs{\bm{n}}$ and finish the proof. 

    To see that $\Psi$ is monotone, suppose $\Psi'(\varphi) = 0$ for some angle $\varphi$. %Then the derivative of $B(e^{i\theta})$ vanishes at $\varphi$ as well. 
    By differentiating $B(e^{i\theta}) = e^{i\Psi(\theta)}$ we obtain
    \begin{equation}\label{eq:derivative of B}
        \frac{d}{d\theta}(e^{i\theta}P(e^{i\theta}))\Bigg|_{\theta = \varphi}P^*(e^{i\varphi})-e^{i\varphi}P(e^{i\varphi})\frac{d}{d\theta}P^*(e^{i\theta})\Bigg|_{\theta = \varphi} = 0.
    \end{equation}
    With $\tau = -B(e^{i\varphi})$ it follows from \eqref{eq:derivative of B} that
    \begin{align*}
        \frac{d}{d\theta}P^{(\tau)}(e^{i\theta}) & = \frac{d}{d\theta}(e^{i\theta}P(e^{i\theta}))\Bigg|_{\theta = \varphi} + \tau \frac{d}{d\theta}P^*(e^{i\theta})\Bigg|_{\theta = \varphi} \\ & = \frac{d}{d\theta}P^*(e^{i\theta})\Bigg|_{\theta = \varphi}\brkt{B(e^{i\varphi}) + \tau} \\ & = 0.
    \end{align*}
    Hence $P^{(\tau)}$ has a double zero, which contradicts Theorem \ref{thm:at zeros}-\ref{thm:angelesco zeros}.
\end{proof}

Now we turn to the Hermite--Pad\'e polynomials $\Phi_{\bm{n},\bm{m}}$ of \cite{KVHP}. The previous result solves the case $\bm{m} = \bm{n}$. Through a modification of the above proof we can now obtain normality in the case $\bm{m} = \bm{n}\pm\bm{e}_k$. 

\begin{thm}\label{thm:normality of hp polynomials}
    Let $\bm{\mu}$ be an Angelesco system or an AT system. Then the multi-indices $(\bm{n},\bm{n}+\bm{e}_j)$ and $(\bm{n}+\bm{e}_j,\bm{n})$ are normal for each $j = 1,\dots,r$. %Furthermore, $\Phi_{\bm{n},\bm{n}+\bm{e}_j}$
\end{thm}
\begin{proof}
    By Proposition \ref{prop:normality criteria hp} it suffices to show that the polynomial $P(z) = z^{\abs{\bm{n}}+1}\Phi_{\bm{n},\bm{n}+\bm{e}_j}(z)$ has $2\abs{\bm{n}}+1$ zeros (with each zero counted according to its multiplicity), given the orthogonality relations and degree restrictions of $\Phi_{\bm{n},\bm{n}+\bm{e}_j} \not\equiv 0$. Note that $\Phi(z) = z^{1/2}\Phi_{\bm{n},\bm{n}+\bm{e}_j}(z)$ satisfies the orthogonality relations
    \begin{equation}
        \int \Phi(z)z^{-p}d\mu_j(z) = 0, \quad p = -n_j+1/2,\dots,n_j-1/2, \quad j = 1,\dots,r. 
    \end{equation}
    
    Now consider the function
    \begin{equation}
        X_{2\bm{n}}^{(\tau)}(z) = z^{1/2}\Phi_{\bm{n},\bm{n}+\bm{e}_j}(z) + \tau z^{-1/2}\Phi_{\bm{n},\bm{n}+\bm{e}_j}^\sharp(z). 
    \end{equation}
    Observe that $X^{(\tau)}_{2\bm{n}}(z)\in\operatorname{span}\{z^p\}_{p=-|\bm{n}|-1/2}^{|\bm{n}|+1/2}$ and satisfies \eqref{eq:orthoPOPUC}--\eqref{eq:beta invariance} for index $2\bm{n}$. Note that its highest coefficient is not $1$, unlike the polynomial defined in~\eqref{eq:paraorthogonal polynomials defn}.
    
    By Proposition \ref{prop:popuc normality}, the polynomial $P^{(\tau)}(z) = z^{\abs{n}+1/2}X_{2\bm{n}}^{(\tau)}(z)$ has degree $2\abs{\bm{n}}+1$. We also have
    \begin{equation}
        P^{(\tau)}(z) = P(z) + \tau P^*(z),
    \end{equation}
    where $P^*(z) = z^{2\abs{\bm{n}}+1}\overline{P(1/\bar{z})}$. From here we can repeat the ideas from the proof of Theorem \ref{thm:zeros in unit disc}.
    
    First, it follows similarly that $P$ cannot have zeros on $\d\D$, and then we can introduce the Blaschke product $B(z) = P(z)/P^*(z)$. We have $B(e^{i\theta}) = e^{i\Psi(\theta)}$ for the continuous phase function 
    \begin{equation}
        \Psi(\theta) = 2\sum_{j = 1}^{N}\brkt{-\frac{\theta}{2} + \operatorname{Arg}(e^{i\theta}-z_j)}, \quad \theta \in [-\pi,\pi),
    \end{equation}
    where $N = \deg{P} \leq 2\abs{\bm{n}}+1$ and $z_1,\dots,z_N$ are the zeros of $P$. The argument principle then gives $\Psi(\pi-)-\Psi(-\pi) = 2\pi(n_+-n_-) = \pm2\pi(2\abs{\bm{n}}+1)$, where the sign depends on whether $\Psi$ is increasing or decreasing. Depending on the sign we then have either $2\abs{\bm{n}}+1$ zeros in $\D$ or $2\abs{\bm{n}}+1$ zeros in $\C\setminus\bar\D$. Either case proves $\deg{P} = 2\abs{\bm{n}}+1$, which completes the proof. 
    \end{proof}

\bigskip\bigskip

\bibsection

%\end{thebibliography}

%%%%%%%%%%%%%%%%%%%%%%%%%%%%%%%%%%%%%%%%%%%%%%%%%%%%%%%%

\end{document}